\documentclass[a4paper,12pt,reqno]{amsart}
\usepackage{amssymb,amsfonts,amsthm,amsmath,url}
\usepackage[margin=0.9in]{geometry}
\usepackage{mathtools}
\usepackage{amscd}
\usepackage{setspace}
\usepackage{version}
\usepackage[pdftex,bookmarks=true,bookmarksnumbered,colorlinks=true]{hyperref}
\usepackage[hyperpageref]{backref}
\mathtoolsset{showonlyrefs}
\numberwithin{equation}{section}
\theoremstyle{plain}
\newtheorem{theorem}[subsection]{Theorem}

\newtheorem{lemma}[subsection]{Lemma}

\newtheorem{remark}[subsection]{Remark}
\newtheorem{definition}[subsection]{Definition}

\newcommand			{\md}		[1]	{\ensuremath{(\operatorname{mod}\, #1)}}
\newcommand			{\mdsub}		[1]	{\ensuremath{(\mbox{\scriptsize mod}\, #1)}}

\newcommand\N		{\mathbb{N}}
\newcommand\Z		{\mathbb{Z}}
\newcommand\R		{\mathbb{R}}
\newcommand\F		{\mathbb{F}}

\newcommand\p		{\mathfrak{p}}

\newcommand\Sym		{\operatorname{Sym}}
\newcommand\GL		{\operatorname{GL}}
\newcommand\ord		{\operatorname{ord}}
\newcommand\li		{\operatorname{li}}

\newcommand\mat		{\operatorname{Mat}}

\renewcommand{\leq}	{\leqslant}
\renewcommand{\geq}	{\geqslant}

\DeclareMathOperator*{\Osum}{\sum{}^*}

\makeatletter

\renewcommand{\theenumi}{(\roman{enumi})}

\renewcommand{\p@enumii}{\theenumi.}
\makeatother
\newtheoremstyle{named}{}{}{}{}{\bfseries}{.}{.5em}{\thmnote{#3's }#1}
\theoremstyle{named}

\newtheorem{ack}{Acknowledgements}

\begin{document}
\title[On the density of Singer cycles in $\GL_n(q)$]{On the distribution of the density of maximal order elements in general linear groups}
\author{S. Aivazidis}
\address{School of Mathematical Sciences, Queen Mary, University of London, London E1 4NS}
\email{s.aivazidis@qmul.ac.uk}
\author{E. Sofos}
\address{School of Mathematics, University of Bristol, Bristol BS8 1TW}
\email{efthymios.sofos@bristol.ac.uk}
\thanks{The first author acknowledges financial support from the N. D. Chrysovergis endowment under the auspices of the National Scholarships Foundation of Greece. The second author was supported by EPSRC grant EP/H005188/1.
}
\begin{abstract}
In this paper we consider the density of maximal order elements in $\GL_n(q)$. 
Fixing any of the rank $n$ of the group, the characteristic $p$ or the degree $r$ 
of the extension of the underlying field $\F_q$ of size $q=p^r$, we compute 
the expected value of the said density and establish that it follows a distribution law.

\end{abstract}

\maketitle

\section{Introduction}
In a series of papers stretching from 1965 to 1972, and beginning with \cite{erdos}, Erd\H{o}s and Tur\'{a}n examined in detail various questions of a statistical nature regarding the symmetric group. One might ask whether similar work can be done for other classes of groups; indeed, the most natural next candidate is the general linear group $\GL_n(q)$ of 
$n \times n$ matrices over the finite field $\F_q$. A notable difference between the two classes of groups is the dependence of the latter on more than one parameter. 
While $\Sym \left( \Omega \right)$ is completely determined (up to isomorphism) by the cardinality of the set $\Omega$, the groups $\GL_n(q)$ require 3 variables for their definition: the rank $n$ and the size $q$ of the underlying field, which, in turn, depends both on the characteristic of the field and on the degree of the extension. Stong~\cite{stong1} considered the average order of a matrix in $\GL_n(q)$ for fixed $q$ and varying $n$ and proved that 
\begin{equation}
\log \nu_n = n \log q - \log n + o\left( \log n \right),
\end{equation}
where
\begin{equation}
\nu_n = \frac{1}{\left| \GL_n(q) \right|}\sum\limits_{ A \in \GL_n(q)}\ord(A).
\end{equation}

Our purpose here is to address a question which is more limited in scope, but rather different in nature from Stong's investigations. In particular, we shall consider elements of maximal order in $\GL_n(q)$, also known as Singer cycles, and examine in detail the mean density of those elements in $\GL_n(q)$. Fixing any two of the parameters $n$, $p$, $r$ and letting the remaining one vary accordingly, we show that the density of Singer cycles follows a distribution law and provide its expected value.
It is straightforward to show that the maximal order of an element in $\mathrm{GL}_n(q)$ is $q^n-1$. In section~\ref{singers} we shall do this after establishing existence of the said elements and obtain the formula
\begin{equation}
\frac{\left| \GL_n(q) \right|}{q^n-1}  \frac{\phi \left( q^n-1\right)}{n}
\end{equation}
for the number of Singer cycles, where $\phi$ is the usual Euler function. The core of the paper is concerned with their density function 
\begin{equation}\label{densityf}
\p_n(q) \coloneqq \frac{1}{n}  \frac{\phi(q^n-1)}{q^n-1},
\end{equation}
i.e., the probability that an element has maximal order in $\GL_n(q)$.

We are interested in the distribution of $\p_n({q})$ in the interval $(0,\frac{1}{n})$.
If we fix $n$ and let $q\to \infty$ through prime powers, we see that $\lim_q \p_n({q})$ does not exist. It thus makes sense to examine its average value instead. 
In Theorem~\ref{main}~\ref{111} we provide the answer to this question. However, the average is not greatly affected by the values that $\p_n({q})$ assumes when $q$ varies through genuine powers of primes. We therefore investigate the average of $\p_n({p^r})$ for fixed $n,p$ and for varying $r$. This is the content of Theorem~\ref{main}~\ref{112}. Lastly, by similar methods we provide the average value for the case
that Stong deals with, that is when the field is fixed and the rank $n$ varies, and that is the content of Theorem~\ref{main}~\ref{113}. We notice here that similar questions have been considered previously in the case where one has the multiplicative group $\Z/m\Z$ instead of the general linear group. See for example~\cite{li} and the survey~\cite{pompom}.

Before we state our first theorem, let us introduce some relevant terminology. 
For a prime $p$ and an integer $a$ coprime to $p$, we let $\ell_p(a)$ denote its multiplicative order $\md{p}$, that is the least positive integer $k$ for which $p^k\equiv 1 \md{a}$.
Define for a prime $p$ and integer $r$ the following series
\begin{equation}\label{ff}
\p(p,r) \coloneqq \Osum\limits_{m \in \N}\frac{\mu(m)}{m}\frac{\gcd\left( \ell_p(m),r \right)}{\ell_p(m)},
\end{equation}
Here and throughout this paper $\Osum$ denotes a summation over those positive integers $m$ that are coprime to $p$. We will show in Lemma~\ref{nerdos}~~\ref{443} that this series is convergent. 
We also define for an integer $n$ the quantity
\begin{equation}\label{product}
\p_n \coloneqq \frac{1}{n}\prod\limits_{p}\left(1-\frac{\gcd\left(p-1,n\right)}{p(p-1)} \right).
\end{equation}
Notice that $\p_1$ is the so--called Artin constant, arising in Artin's primitive root conjecture. 
The infinite product in \eqref{product} converges since for fixed $n$ one has $\gcd\left(p-1,n\right) \leq n$, hence 
\begin{equation}
\sum\limits_{p}\frac{\gcd\left(p-1,n\right)}{p(p-1)} \leq n\sum\limits_{p}\frac{1}{p(p-1)} <  \infty.
\end{equation}
Our first result is the following.

\begin{theorem}\label{main}
Let $x \in \mathbb{R}_{\geq 1}$ and denote the cardinality of the powers of primes below $x$ by $Q\left(x\right)$.
\begin{enumerate}
\item \label{111}	For any fixed $n \in \N$ and any $A \in \mathbb{R}_{>1}$ one has 
				\begin{equation}	
				\frac{1}{Q\left( x \right)}\sum_{q \leq x}\p_n(q) = \p_n+O_{n,A}\left(\frac{1}{\left(\log x\right)^A}\right),
				\end{equation}
				where the summation is taken over powers of primes.  	\\
\item \label{112}	For any fixed prime $p$ and $n \in \N$ one has 	
				\begin{equation}	
				\frac{1}{x}\sum_{r \leq x}\p_n(p^r) = \frac{1}{n}\p(p,n)+O\left(\frac{\log\left( x n \log p\right)}{xn}\right),
				\end{equation}
				where the implied constant is absolute.				\\
\item \label{113}	For any fixed $q=p^r$ one has 				
				\begin{equation}	
				\frac{1}{\log x}\sum_{n \leq x}\p_n(q) = \p(p,r)+O\left(\frac{\tau(r)\log \left(r\log p\right)}{\log x}\right),
				\end{equation}
				where the implied constant is absolute
				and $\tau$ denotes the divisor function.
\end{enumerate}
\end{theorem}
\begin{remark}
The above theorem reveals the noteworthy fact that the density of Singer cycles in $\GL_n(q)$ is approximated on average by a constant multiple of $\frac{1}{n}$
when one allows any of the underlying parameters to vary.
\end{remark}

The special case corresponding to $n=1$ in Theorem~\ref{main}~\ref{111}
has been dealt with previously by Stephens~\cite[Lemma 1]{step} and in an equivalent 
form by Pillai~\cite[Theorem 1]{pill}. Following the proof of Theorem~\ref{main}~\ref{111}, 
we show that the error of approximation can be substantially improved to 
\begin{equation}
O_n\left(\frac{(\log x)^{\tau(n)+2}}{\sqrt{x}}\right)
\end{equation} 
under the assumption of the Generalised Riemann Hypothesis.

One should notice that the sequence $\p_n(q)$ oscillates wildly around its mean value in all cases of Theorem~\ref{main} and 
it would therefore be interesting to attain information regarding
the nature of its distribution. It thus makes sense to examine how $\p_n(q)$ distributes over subintervals of $\left(0,\frac{1}{n}\right)$.
To that end let us recall some standard definitions from Probabilistic Number Theory (see~\cite[Chapter III]{tene} for a detailed discussion).
Let  $b_n$ be a sequence of real numbers
and define for
$x,z \in \R$, 
the frequencies $\nu_x$  as follows:
\begin{equation}
\nu_x \left( n ; b_n \leq z \right)		\coloneqq		\frac{\left|\left\{n\leq x: b_n\leq z\right\}\right|}{x}.
\end{equation}
Similarly denote
\begin{equation}
\nu_x\left(q;b_q \leq z \right)			\coloneqq		\frac{\left|\left\{q\leq x: q \ \text{is a prime power} ,b_q\leq z\right\}\right|}{Q(x)}.
\end{equation} 
We say that \textit{the frequencies $\nu_x$ converge to a limiting distribution as 
$x \to \infty$,}
if for any $z$ in a certain dense subset $E \subset \R$, the following limit exists
\begin{equation}
\lim_{x \to \infty}	\nu_x\left(n;b_n \leq z\right),
\end{equation}
and furthermore, denoting its value by $F(z)$, that one has 
\begin{equation}
\lim_{\substack{z \to \alpha \\z \in E}}F(z)=
\left\{
        \begin{array}{ll}
            1,\,\ \text{if $\alpha = +\infty$}\\
            0,\,\ \text{if $\alpha = -\infty$}
        \end{array}
    \right.
.\end{equation}
Thus the existence of a limiting distribution should be interpreted as an equidistribution of $b_n$ with respect to some measure.
We are now ready to state our second theorem.
\begin{theorem}\label{main2}
The frequencies $\nu_x$, with respect to any of the involved parametres, converge to a limiting distribution. More precisely: 
\begin{enumerate}
\item 	\label{main2i}	For fixed $n \in \N$ the frequencies $\nu_x \left(q;\p_n(q) \leq z\right)$ converge to a continuous limiting distribution.
\item 	\label{main2ii}	For fixed prime $p$ and integer $n$ the frequencies $\nu_x \left(r;\p_n(p^r) \leq z\right)$ converge to a limiting distribution.
\item 	\label{main2iii}	For fixed prime power $q=p^r$ the frequencies $\nu_x(n;\p_n(q) \leq \frac{z}{n})$ converge to a limiting distribution.
\end{enumerate}
\end{theorem}

Let us put Theorem~\ref{main2} in context.
The easier problem of determining the frequencies 
$\nu_x \left(n;\phi(n)/n \leq z\right)$, where one 
ranges general integers $n$, has received much attention. 
Shoenfeld~\cite{schoen1} proved that these frequencies converge 
to a limiting distribution, say $F(z)$, and that $F(z)$ is continuous.
In a subsequent paper~\cite{schoen2} he showed that $F(z)$ is strictly 
increasing. Erd\H{o}s~\cite{paulie} discovered the following asymptotic 
expression:
\begin{equation} \label{sofos} 
F(1-\epsilon)=1-\frac{e^{-\gamma}}{\log \frac{1}{\epsilon}}
+O\left(\frac{1}{\left(\log \frac{1}{\epsilon}\right)^2}\right)
\end{equation}
uniformly for all $\epsilon \in (0,1)$, where $\gamma$ is the Euler constant.
Toulmonde~\cite{toul} proved 
that we can in fact get a more precise expression
in the right-hand-side of~\eqref{sofos}, involving an asymptotic expansion in negative powers of
$\frac{1}{\log\left(1/\epsilon\right)}$ in place of the error term.
The problem however becomes less easy when we range over 
powers of primes rather than general integers.
The special case of part~\ref{main2i} of Theorem~\ref{main2} corresponding 
to $n=1$ has been handled by K{\'a}tai~\cite{katari}, who showed 
that the limiting distribution exists and is continuous. 
Deshouillers and Hassani~\cite{soublakakakia} proved that this limiting distribution possesses infinitely many points of non--differentiability.
It would be desirable to have analogues of \eqref{sofos}
for part~\ref{main2i} of Theorem~\ref{main2}, even in the case $n=1$,
a problem which is equivalent to obtaining an analogue of \eqref{sofos}
regarding the limiting distribution of $\displaystyle\frac{\phi(p-1)}{p-1}$,
as $p$ ranges in primes.

A few remarks about the proofs of the two theorems are in order. Part~\ref{111} of Theorem~\ref{main} is proved via 
splitting a certain sum over moduli into two ranges according to the size of the moduli. 
The contribution coming from the range corresponding to small moduli
will give the main term after an application of the Bombieri--Vinogradov Theorem~\ref{Bombieri--Vinogradov} and 
the range corresponding to large moduli will be shown to give a negligible contribution compared to the main term.

The proofs of Theorem~\ref{main}~\ref{112} and Theorem~\ref{main}~\ref{113} are rather similar; both are based on 
Lemma~\ref{phi}, which is similar to~\cite[Theorem 3]{shpar}. There, however, the dependence of the error term 
on the underlying parameters is not best possible and we will attempt to remedy this. The proof of Lemma~\ref{phi} 
is based on Lemma~\ref{nerdos}, where all sums appearing in its statement resemble Romanoff's series
\begin{equation}
\Osum_{m \in \N}\frac{1}{m\ell_p(m)},
\end{equation}
where $p$ is a fixed prime. We will bound such sums by introducing the auxiliary quantity $E_p(x,d)$, defined later 
in the proof of Lemma~\ref{nerdos}, which is a trick introduced by Erd\H{o}s~\cite{roman}.

The proof of Theorem~\ref{main2}~\ref{main2i} is conducted via restricting attention to prime numbers below $x$ and then 
applying a theorem in~\cite{tole} regarding distribution laws of $f\left(\left|g(p)\right|\right)$, for an additive function $f$
and a polynomial $g \in \Z[x]$.
The proofs of parts~\ref{main2ii}~and~\ref{main2iii} of Theorem~\ref{main2} are again quite similar. Here, however, the sequence involved is not additive. Therefore 
we have to use a result for distribution laws of general arithmetic functions.

The rest of the paper is organised as follows. 
In section~\ref{singers} some background on Singer cycles is given, resulting in the explicit estimation of $\p_n(q)$.
In section~\ref{preliminaries} we provide some analytic tools that will be used later in the proof of Theorem~\ref{main}. 
In section~\ref{lemmata} some auxiliary lemmata regarding upper bounds of sums of certain arithmetic functions are provided. 
Finally, in sections~\ref{mainproof} and~\ref{mainproof2} we provide the proof of Theorems~\ref{main} and~\ref{main2} respectively.

\subsection{Notation}
Throughout this paper $p$ will denote a prime and $q$ a (not necessarily proper) power of a prime.
\begin{enumerate}
\item 	The notation $\sum_{p}$ is understood to be a sum taken over primes; similarly, $\sum_q$ should be read as a sum taken over powers of primes, and the same principle 			applies to products. As already mentioned, $\Osum$ denotes a summation over positive integers that are coprime to $p$. \\
\item 	We shall write $p^\lambda\|n$ for a prime $p$ and positive integers $\lambda$, $n$ if $p^\lambda \mid n$ and $p^{\lambda+1}\nmid n$. \\
\item		For the real functions $f(x)$, $g(x)$, defined for $x>0$, the notation $f(x)=O\left(g(x)\right)$ (or, equivalently, $f(x) \ll g(x)$) means that there exists an absolute constant 
		$M>0$ independent of $x$, such that $\left|f(x)\right| \leq M \left|g(x)\right|$ for $x>0$. We shall write $f(x)=g(x)+O\left(h(x)\right)$ if $(f-g)(x)=O\left(h(x)\right)$. When the implied constant depends on a set of parameters $\mathcal{S}$, we shall write $f=O_{\mathcal{S}}(g)$, or $f \ll_{\mathcal{S}}g$. If no such subscript appears, then the implied constant is absolute.\\
\item 	As usual, we let $\mu(n)$ the M\"{o}bius function, $\sigma(n)$ the sum of divisors of $n$,
		and $\Lambda(n)$ the von Mangoldt function. This, we recall, is defined as
		\begin{equation}
		\Lambda(n)=
		\begin{cases}
		\log p,   & \quad\text{if}~n~\text{is a power of a prime}~p, \\
       		0,   & \quad\text{otherwise}.
		\end{cases}
		\end{equation}
\end{enumerate}
\section{Singer cycles}\label{singers}
An element of order $q^n-1$ in $\GL_n(q)$ is called a \textit{Singer cycle}. That Singer cycles always exist can be seen as follows.

Let $\F_{q^n}$ be the $n$-degree field extension of $\F_q$, and let $\alpha$ be a primitive element of $\F^{\ast}_{q^n}$.
The map 
\begin{equation}
s : \F_{q^n} \to \F_{q^n}, \,\ x \mapsto \alpha x
\end{equation} 
is $\F_q$-linear and invertible. Further, the order of $s$ is equal to that of $\alpha$ in $\F^{\ast}_{q^n}$, that is, $q^n-1$.

In fact the integer $q^n-1$ is maximal among possible element orders in $\GL_n(q)$.
To see why that must be the case, consider the algebra $\mat_n(q)$ of all $n \times n$ matrices over $\F_q$ and note that each element $A \in \GL_n(q)$ generates 
a subalgebra $\F_q \left[ A\right] \subseteq \mat_n(q)$. The Cayley-Hamilton theorem then ensures that $\dim \F_q \left[ A \right] \leq n$, thus $o(A) \leq q^n-1$ for all $A \in \GL_n(q)$. 

Our claim now is that $o(A)= q^n-1$ if and only if the minimal polynomial $m_A(x)$ of $A$ is primitive of degree $n$ (recall that $f$ is a primitive polynomial if and only if any root of $f$ in the splitting field of $f$ generates the multiplicative group of that field). For this we shall need the following~\cite[Lemma 3.1]{ffields}.
\begin{lemma}
Let $f \in \F_q[x]$ be a polynomial of degree $m \geq 1$ with $f(0) \neq 0$. Then there exists a positive integer $e \leq q^m -1$ such that $f(x) \mid x^e-1$.
\end{lemma}
The least such $e$ is called the \textit{order} of $f$ and is denoted by $\ord (f)$. Clearly $A^k = I_n$ if and only if $m_A(x) \mid x^k -1$, $k \in \N$. 
Thus\footnote{This yields yet another proof that the order of a matrix is at most $q^n-1$.}
\begin{equation}
o(A) = \ord \left( m_A \right).
\end{equation}
From Lidl and Niederreiter~\cite[Theorem 3.3]{ffields} we know that the order of an irreducible polynomial of degree $n$ over $\F_q[x]$ is equal to the order of any of its roots 
in $\F^{\ast}_{q^n}$. Thus  $\ord \left( m_A \right) = q^n-1$ if and only if $m_A$ is primitive and $\deg m_A = n$, as wanted. Note that the above argument shows that the minimal and the characteristic polynomial of a Singer cycle coincide. 

We shall now give a proof for the number of Singer cycles as an application of a theorem of Reiner~\cite[Theorem 2]{reiner}, but it ought to mentioned that the same formula can be obtained via the familiar Orbit--Stabiliser Theorem (a Singer cyclic subgroup has index exactly $n$ in its normaliser). Reiner computes the number of (not necessarily invertible) $n \times n$ matrices with entries in the finite field $\F_q$ having given characteristic polynomial.
Let $R_n$ denote the ring of all $n \times n$ matrices with entries in~$\F_q$, and define $F(u,r) = \prod\limits_{i=1}^{r} \left( 1-u^{-i} \right)$, where $F(u,0)=1$. 

\begin{theorem}[Reiner~\cite{reiner}]
Let $g(x) \in \F_q[x]$ be a polynomial of degree $n$, and let
\begin{equation}	
g(x) = f_1^{n_1}(x) \cdots f_k^{n_k}(x)
\end{equation}
be its factorisation in $\F_q[x]$ into powers of distinct irreducible polynomials $f_1(x), \dots f_k(x)$. Set
$d_i \coloneqq	\deg \left(f_i(x)\right)$, $1 \leq i \leq k$. Then the number of matrices $X \in R_n$ with characteristic polynomial $g(x)$ is
\begin{equation}	
q^{n^2-n}\frac{F(q,n)}{\prod\limits_{i=1}^{k}F\left( q^{d_i},n_i\right)}.
\end{equation}
\end{theorem}
Now take $k=1$, $n_1=1$, and $d_1=n$ in the above formula and notice that a matrix whose characteristic polynomial is thus parameterised is necessarily invertible. 
In particular there are 
\begin{equation}	
q^{n^2-n}\frac{F(q,n)}{F\left( q^n,1\right)} = \frac{\left| \GL_n(q) \right|}{q^n-1}
\end{equation}
such matrices. Since there are precisely 
\begin{equation}
\frac{\phi \left( q^n-1\right)}{n}
\end{equation}
primitive (thus also irreducible) polynomials of degree $n$ in $\F_q [x]$, the following has been proved.
\begin{lemma}
The number of Singer cycles in $\GL_n(q)$ is given by the formula
\begin{equation}
\frac{\left| \GL_n(q) \right|}{q^n-1}  \frac{\phi \left( q^n-1\right)}{n}.
\end{equation}
\end{lemma}
\section{Preliminaries}\label{preliminaries}
In this section we recall standard results regarding the distribution of primes in arithmetic progressions. 
Denote by $\pi(x)=\sum_{p \leq x}1, x\geq 3$ the number of primes less than or equal to $x$ and by $\li(x)$ the logarithmic integral
\begin{equation}
\li(x) \coloneqq \int_{2}^x \frac{\mathrm{d}t}{\log t}.
\end{equation}
The Prime Number Theorem states that for each $A>0$ and for all $x \geq 3$ one has 
\begin{equation}
\pi(x)=\li(x)+O_A\left(\frac{x}{\left(\log x\right)^A}\right).
\end{equation} 
For coprime integers $m,a$ and all $x \geq 3$, define
\begin{equation}\label{psidefinition}
\psi(x;m,a)	\coloneqq	\sum	\limits_{\substack{n \leq x \\ n \equiv a \mdsub m}} \Lambda(n).
\end{equation}
The Bombieri--Vinogradov theorem \cite[\S 28]{davenport} then states that:
\begin{theorem}\label{Bombieri--Vinogradov} 
For any fixed $A>0$ and for all $x \geq 3$ one has
\begin{equation}
\sum_{m \leq \sqrt{x} / \left(\log x\right)^A}	\max \limits_{\substack{a \md{m} \\ (a,m)=1}}	\left|\psi(x;m,a)-\frac{x}{\phi(m)}\right| \ll_A\frac{x}{\left(\log x\right)^{A-5}}.
\end{equation}
\end{theorem}
We shall also make use of the following lemma.
\begin{lemma}\label{pt} 
Let $a,m$ be coprime integers. Then one has the following estimate
\begin{equation}
\sum_{\substack{p\leq x\\p\equiv a \md{m}}}\frac{1}{p}=\frac{1}{\phi(m)}\log \log x+O_m(1),
\end{equation}
for $x\geq 3$. 
\end{lemma}
\begin{proof}
Corollary 4.12(c) in Montgomery--Vaughan~\cite{mova} yields
\begin{equation}
\sum_{\substack{p\leq x\\p\equiv a \md{m}}}\frac{1}{p}=\frac{1}{\phi(m)}\log \log x+O_m\left(1+\sum_{\chi\neq \chi_0 \md{m}}\log \left|L(1,\chi)\right|\right),
\end{equation}
where the summation $\sum_{\chi\neq \chi_0 \md{m}}$ is taken over non--trivial characters $\md{m}$. Using the fact that there are finitely many such characters and that 
$L(1,\chi)\neq 0$ for each such character proves the lemma.
\end{proof}

\section{Lemmata}\label{lemmata}
In this section we introduce a certain arithmetical function $\rho_n(m)$ and provide an explicit expression when $m$ is square--free. We then obtain upper bounds for sums that involve this function. At the end of the section we give some auxiliary lemmata regarding sums that involve the multiplicative order function $\ell_p(n)$.

Let $n \in \N$ be a fixed positive integer, and define the function $\rho_n : \N \to \N $ via the rule
\begin{equation}\label{rho0} 
m		\mapsto		\left|	\left\{	a \ \md{m}:a^n \equiv 1 \ \md{m}	\right\} \right|.
\end{equation} 
It is a direct consequence of the Chinese Remainder Theorem that $\rho_n$ is multiplicative.
The verification of the following lemma is left to the reader.
\begin{lemma}\label{rho} 
For all primes $p$ and for all positive integers $n$ one has 
\begin{equation}\rho_n ( p )=\gcd\left(p-1,n\right).\end{equation}
\end{lemma}

\begin{lemma}\label{omeg}
One has the following bounds, valid for all $n \in \N$ and for all $x \geq 3$
\begin{enumerate}
\item 	\label{421}	$\sum\limits_{m \leq x} \mu^2(m)\frac{\rho_n(m)}{m}\ll_n (\log x)^{\tau(n)}$, 
\item 	\label{422}	$\sum\limits_{m>x} \mu^2(m)\frac{\rho_n(m)}{m^2} \log m\ll_n \dfrac{1}{x}(\log x)^{\tau(n)+1}$.
\end{enumerate}
\end{lemma}
\begin{proof}
\ref{421} We begin by noticing that if $f:\N \to \R_{\geq 0}$ is a multiplicative function then 
\begin{equation}
\sum_{m\leq x} \mu^2(m)f(m) \leq \prod_{p\leq x} \left(1+f\left(p\right)\right).
\end{equation}
Setting $f(m)=\dfrac{\rho_n(m)}{m}$ in the above relation yields
\begin{equation}
\sum_{m\leq x} \mu^2(m)\frac{\rho_n(m)}{m}\leq\prod_{p\leq x} \left(1+\frac{\gcd\left(p-1,n\right)}{p}\right)
\end{equation}
for all $x \geq 3$. The inequality $1+t \leq e^t$, valid for any $t>0$, shows that 
\begin{equation}
\prod_{p\leq x} \left(1+\frac{\gcd\left(p-1,n\right)}{p}\right)\leq \exp\left(\sum_{p\leq x}\frac{\gcd\left(p-1,n\right)}{p}\right).
\end{equation}
Using the identity  
\begin{equation}
\label{selbergios} 
\gcd(a,b)=\sum_{\substack{d \mid a \\ d \mid b}}\phi(d),
\end{equation}
valid for all $a,b \in \N$, yields 
\begin{equation}
\sum_{p\leq x}\frac{\gcd\left(p-1,n\right)}{p} = \sum_{d \mid n}\phi(d)\left(\sum_{\substack{p\leq x\\ p\equiv 1 \md{d}}}\frac{1}{p}\right).
\end{equation}
Using Lemma~\ref{pt} for each inner sum gives 
\begin{equation}
\sum_{p\leq x}\frac{\gcd\left(p-1,n\right)}{p} = \tau(n)\log \log x+O_n(1),
\end{equation}
which proves that
\begin{equation}
\sum_{m\leq x} \mu^2(m)\frac{\rho_n(m)}{m}\ll_n(\log x)^{\tau(n)}.
\end{equation}
\ref{422} Splitting the range of summation in disjoint intervals gives
\begin{equation}
\sum\limits_{m>x}\mu^2(m)\frac{\rho_n(m)}{m^2} \log m = \sum_{i=0}^{\infty} \sum_{xe^i<m\leq x e^{i+1}} \mu^2(m)\frac{\rho_n(m)}{m^2} \log m.
\end{equation}
Noticing that
\begin{equation}
\sum_{xe^i<m\leq x e^{i+1}} \mu^2(m)\frac{\rho_n(m)}{m^2} \log m \leq \frac{\log(x e^{i+1})}{xe^i} \sum_{m\leq x e^{i+1}} \mu^2(m)\frac{\rho_n(m)}{m} 
\end{equation}
and applying part~\ref{421} yields
\begin{equation}
\sum\limits_{m>x} \mu^2(m)\frac{\rho_n(m)}{m^2} \log m \ll_n \frac{1}{x} \sum_{i=0}^{\infty}\frac{(\log(xe^{i+1}))^{\tau(n)+1}}{e^i}.
\end{equation}
Using the inequality $\log(ab) \leq (\log a)(\log b)$, valid for all $a , b > e$ gives
\begin{equation}
\sum \limits_{m>x} \mu^2(m)\frac{\rho_n(m)}{m^2} \log m \ll_n \frac{(\log x)^{\tau(n)+1}}{x} \sum_{i=0}^{\infty} \frac{(i+1)^{\tau(n)+1}}{e^i}.
\end{equation}
This proves the assertion of the lemma since the series 
\begin{equation}
\sum_{i=0}^{\infty} \frac{(i+1)^{\tau(n)+1}}{e^i}
\end{equation}
is convergent.
\end{proof}
We record here for future reference a familiar lemma which allows us to translate 
information about the asymptotic behaviour of weighted sums by $\Lambda(k)$ into 
one regarding unweighted sums.
\begin{lemma}\label{gyros}
Let $b_k$ be a sequence of real numbers and define for any $x \in \R_{>1}$,
\begin{equation}
\beta(x) \coloneqq \max\left\{ \left| b_k \right| : 1\leq k \leq x\right\}.
\end{equation}
Then one has the following estimates.
\begin{enumerate}
\item 	\label{431}	For any $x \in \R_{\geq 2}$, $\sum_{q \leq x}b_q=\sum_{p \leq x}b_p+O\left(\sqrt{x}\dfrac{\beta(x)}{\log x}\right)$.
\item		\label{432} 	Let $c$, $A$ be positive constants such that $\sum_{p\leq x}b_p \log p = cx+O\left(\dfrac{x}{(\log x)^A}\right)$, for all $x \geq 3$. Then 
					$\sum_{p \leq x} b_p=c \li(x)+O\left(\dfrac{x}{(\log x)^{A}}\right)$, for all $x \geq 2$.
\end{enumerate}
\end{lemma}
\begin{proof}
\ref{431} Define $k_0 \coloneqq \left[\frac{\log x}{ \log 2}\right]$. We partition the sum $\sum_{q \leq x} b_q$ according to the values of prime powers that $q$ assumes and this yields
\begin{equation} 
\sum_{q \leq x} b_q 		= 	\sum_{k=1}^{k_0}		\sum_{p \leq x^{\frac{1}{k}}} b_{p^k}.
\end{equation}
Now notice that 
\begin{equation}
\sum_{k=2}^{k_0}\sum_{p \leq x^{\frac{1}{k}}} b_{p^k} \ll \beta(x)\left[\pi\left(\sqrt{x}\right)+\left(k_0-2\right) \pi\left(x^{\frac{1}{3}}\right)\right],
\end{equation}
and use the bound $\pi(t) \ll \frac{t}{\log t}$, valid for all $t \in \R_{\geq 2}$, to get
\begin{equation}
\sum_{k=2}^{k_0}\sum_{p \leq x^{\frac{1}{k}}} b_{p^k}\ll \sqrt{x}\frac{\beta(x)}{\log x}.
\end{equation}
The claim follows.

\ref{432} Applying partial summation yields 
\begin{equation}
\sum_{p\leq x}b_p=c \li(x)+O\left(\frac{x}{(\log x)^{A+1}}+\int_2^x\frac{\mathrm{d}t}{(\log t)^{A+2}}\right).
\end{equation}
The use of the following inequalities
\begin{equation}
\int_2^{\sqrt{x}}	\frac{\mathrm{d}t}{(\log t)^{A+2}}		\leq 	\frac{\sqrt{x}}{(\log 2)^{A+2}}, \ \ \ \							
\int_{\sqrt{x}}^{x}	\frac{\mathrm{d}t}{(\log t)^{A+2}}		\leq 	\frac{x-\sqrt{x}}{\left(\log \left(\sqrt{x}\right)\right)^{A+2}},
\end{equation}
concludes the proof of the lemma.
\end{proof}
In the following lemma we record auxiliary bounds that will be needed when we deal with cases~\ref{112}~and~\ref{113} of Theorem~\ref{main}. 
\begin{lemma}\label{nerdos}
Let $p$ be a fixed prime and $d \in \mathbb{N}$. One has the following bounds, uniformly for all $x \geq 1$.
\begin{enumerate}
\item\label{441} 	$\displaystyle\sum_{k \leq x}\Osum_{\substack{m \in \N\\ l_p(m)=kd}}\frac{1}{m}\ll\log (x d \log p)$,
\item\label{442} 	$\displaystyle\sum_{k \geq x}\frac{1}{k}\Osum_{\substack{m \in \N\\ l_p(m)=kd}}\frac{1}{m}\ll\frac{\log (  x d \log p)}{x}$,
\item\label{443} 	the series $\p(p,r)$ defined in~\eqref{ff} converges for each prime $p$ and each $r \in \N$. Further, one has 
				\begin{equation}
				\p(p,r) \ll \tau(r) \log (r \log p),
				\end{equation}
				with an absolute implied constant.
\end{enumerate}
\end{lemma}
\begin{proof}
\ref{441} The integers $m$ taken into account in the inner sum satisfy $p^{kd}\equiv 1 \md{m}$ for some $k\leq x$. Therefore each such $m$ is a divisor of 
\begin{equation}
E_p(x,d) \coloneqq \prod_{k\leq x}(p^{kd}-1).
\end{equation}
Hence the double sum is at most
\begin{equation}
\sum_{m \mid E_p(x)}\frac{1}{m} = \frac{\sigma\left(E_p(x,d)\right)}{E_p(x,d)}.
\end{equation}
We now use the well--known bound 
\begin{equation}
\frac{\sigma(n)}{n}\ll \log \log n
\end{equation}
to deduce that the double sum is 
\begin{equation}
\ll \log \log E_p(x,d).
\end{equation}
One easily sees that 
\begin{equation}
E_p(x,d) \leq \prod_{k\leq x}p^{kd} \leq p^{dx^2},
\end{equation}
which shows that the double sum is $\ll \log \log (p^{dx^{2}}) \ll \log (x d \log p)$, as asserted.

\ref{442} The term corresponding to $k=x$ makes a contribution only when $x$ is an integer,
in which case we get a contribution which is $\ll \frac{1}{x}\log(xd \log p)$, as shown by the first 
part of this lemma. It remains to examine the contribution made by the terms corresponding to $k>x.$
Using partial summation along with part~\ref{441} we deduce that
\begin{equation}
\sum_{k > x}\frac{1}{k}\Osum_{\substack{m \in \N \\ l_p(m)=kd}}\frac{1}{m} \ll \frac{\log(xd\log p)}{x} + \int_x^{\infty}\frac{\log(ud\log p )}{u^2}\mathrm{d}u.
\end{equation}
Alluding to the estimate $\int_{x}^{\infty}(\log u)u^{-2} \mathrm{d}u \ll \log (2x) x^{-1}$, valid for all $ x \geq 1$,
proves our claim.

\ref{443} The identity~\eqref{selbergios}
and 
\begin{equation}
\p(p,r) = \sum_{k=1}^{\infty} \frac{\gcd(k,r)}{k} \Osum_{\substack{m \in \N \\ \ell_p(m)=k}} \frac{\mu(m)}{m}
\end{equation}
show that one has
\begin{equation}
\p(p,r) 
=\sum_{d|r}\frac{\phi(d)}{d}
\sum_{k=1}^{\infty} \frac{1}{k}
\Osum_{\substack{m \in \N \\ \ell_p(m)=kd}} \frac{\mu(m)}{m},
\end{equation}
which is bounded in absolute value by
\begin{equation}
\sum_{d|r}
\sum_{k=1}^{\infty} \frac{1}{k}
\Osum_{\substack{m \in \N \\ \ell_p(m)=kd}} \frac{1}{m}.
\end{equation}
Using
$x=1$ in part~\ref{442} concludes the proof of part~\ref{443}.
\end{proof}
\section{Proof of Theorem~\ref{main}}\label{mainproof}
In this section we prove Theorem~\ref{main}. We begin by proving the auxiliary Lemmata~\ref{psi},~\ref{range}, and~\ref{asympt}
and use them in succession to provide the proof of Theorem~\ref{main}~\ref{111}. We then prove Lemma~\ref{phi} from which we 
deduce the validity of Theorem~\ref{main}~\ref{112} and Theorem~\ref{main}~\ref{113}.

For fixed $n,m\in \N$ and $x\in \R_{>1}$, define the following functions 
\begin{equation}
\label{psi0} 
\Psi_n(x)\coloneqq\sum_{k \leq x}\Lambda(k) \frac{\phi(k^n-1)}{k^n-1},
\end{equation}
and 
\begin{equation}\label{psi1} 
\Psi_n(x;m)\coloneqq
\sum_{\substack{a \mdsub m \\ a^n\equiv 1 \mdsub m}} \psi(x;m,a),
\end{equation}
where $\psi$ was defined in~\eqref{psidefinition}.
\begin{lemma}\label{psi}
For all naturals $n \geq 1$, and for all $x \in \R_{>1}$ one has 
\begin{equation}
\Psi_n(x)=\sum_{m \leq x^n}\frac{\mu(m)}{m} \Psi_n(x;m).
\end{equation}
\end{lemma}
\begin{proof}
The proof follows readily by noticing that 
\begin{equation}\label{penguin}
\frac{\phi(k)}{k}=\sum\limits_{m \mid k}\frac{\mu(m)}{m},
\end{equation} 
and inverting the order of summation.
\end{proof}

\begin{lemma}\label{range} 
For any fixed constant $A>0$ and any fixed $n \in \N$, one has uniformly for all $x\geq 3$
\begin{equation}\label{five}
\sum_{\sqrt{x}/\left(\log x\right)^{A} < m \leq x^n} \frac{\mu(m)}{m} \Psi_n(x;m)  \ll_{A,n} \sqrt{x} \left(\log x\right)^{2+\tau(n)+A}.
\end{equation}
\end{lemma}
\begin{proof}We break the summation over $m$
into the two disjoint intervals
$[\sqrt{x}/\left(\log x\right)^{A} ,x]$
and $\left(x,x^n\right]$. We first deal with the contribution afforded by the latter interval.
We claim that for $m > x$ one has $\psi(x;m,a) \leq \log x$. To see why, note 
that there exists at most one $k$ in $\left[1,x \right]$ such that $k \equiv a \md m$, 
and notice that $\Lambda(k) \leq \log k$, for all $k \in \N$, with equality if 
and only if $k$ is prime. Thus $\psi(x;m,a) \leq \log x$, as wanted.
Recalling the definition of $\rho_n(m)$, given in equation \eqref{rho0},
we get for $m > x$,
\begin{equation}
\Psi_n(x;m) = \sum_{\substack{a\md{m} \\a^n \equiv 1 \md{m}}}\psi(x;m,a) \leq \rho_n(m) \log x.
\end{equation}
Therefore 
\begin{equation}
\left|\sum_{x < m \leq x^n}\frac{\mu(m)}{m}\Psi_n(x;m)\right| \leq \log x \sum_{m \leq x^n} \mu^2(m)\frac{\rho_n(m)}{m}.
\end{equation}
We use the first part of Lemma~\ref{omeg} to conclude that 
\begin{equation}
\left|\sum_{x < m \leq x^n}\frac{\mu(m)}{m}\Psi_n(x;m)\right| \ll_n \left(\log x\right)^{1+\tau(n)}.
\end{equation} 
We proceed by estimating the contribution inherited from $m$
in the range 
$(\sqrt{x}/\left(\log x\right)^{A} ,x]$.
Since 
\begin{equation} 
\sum_{\substack{k \leq x \\ k \equiv a \md{m}}}1 \leq \left[\frac{x}{m}\right]+1 \leq 2\frac{x}{m} 
\end{equation}
for $m \leq x$, we see that
\begin{equation}
\psi(x;m,a)=\sum_{\substack{k\leq x \\ k \equiv a \md{m}}}\Lambda(k)\leq\log x\sum_{\substack{k\leq x \\ k \equiv a \md{m}}}1\leq2 \frac{x}{m}\log x.
\end{equation}
Therefore
\begin{align}
\left|\sum_{\sqrt{x}/\left(\log x\right)^A < m \leq x}\frac{\mu(m)}{m}\Psi_n(x;m)\right|										&\leq 
\sum_{\sqrt{x}/\left(\log x\right)^A < m \leq x}\frac{\mu^2(m)}{m}\sum_{\substack{a\md{m} \\ a^n \equiv 1 \md{m}}}\psi(x;m,a)	\\
																											&\leq
2x \log x\sum_{\sqrt{x}/\left(\log x\right)^A < m \leq x}\mu^2(m)\frac{\rho_n(m)}{m^2}.
\end{align}
Using the second part of Lemma \eqref{omeg} we get 
\begin{equation}
x \log x \sum_{m>\sqrt{x}/\left(\log x\right)^{A}}\mu^2(m)\frac{\rho_n(m)}{m^2} \ll_{n,A}\sqrt{x}\left(\log x\right)^{\tau(n)+A+2},
\end{equation}
thus completing the proof.
\end{proof}

\begin{lemma}\label{asympt} 
For any fixed constant $A>0$ one has uniformly for all $x\geq 3$
\begin{equation}
\sum_{m \leq \sqrt{x}/\left(\log x\right)^{A} }\frac{\mu\left(m\right)}{m} \Psi_n\left(x;m\right) = n \p_n x+ O_{n,A} \left(\frac{x}{\left(\log x\right)^{A-5}}\right).
\end{equation}
\end{lemma}
\begin{proof}
By the definition of $\Psi_n\left(x;m\right)$ we have
\begin{equation}
\sum_{m \leq \sqrt{x}/\left(\log x\right)^{A} }\frac{\mu\left(m\right)}{m} \Psi_n\left(x;m\right)		=
\sum_{m \leq \sqrt{x}/\left(\log x\right)^{A} }\frac{\mu\left(m\right)}{m}\sum_{\substack{a\md{m} \\ a^n \equiv 1 \md{m}}} \psi\left(x;m,a\right),
\end{equation}
which equals
\begin{align}
&\sum_{m \leq \sqrt{x}/\left(\log x\right)^{A}}\frac{\mu\left(m\right)}{m}\sum_{\substack{a\md{m}\\a^n\equiv 1 \md{m}}} \left(\psi\left(x;m,a\right)-\frac{x}{\phi\left(m\right)}\right)\\+x
&\sum_{m \leq \sqrt{x}/\left(\log x\right)^{A} }\frac{\mu\left(m\right)}{m}\sum_{\substack{a\md{m} \\ a^n\equiv 1 \md{m}}} 
\frac{1}{\phi\left(m\right)}\\
&=\mathcal{E}+x\mathcal{M}, \text{say.}
\end{align}
Recalling the definition of $\rho_n\left(m\right)$ (equation~\eqref{rho0}),
one has 
\begin{equation}
\label{abcd}
\left|\mathcal{E}\right| \leq \sum_{m \leq \sqrt{x}/\left(\log x\right)^{A} }
\frac{\rho_n\left(m\right)}{m}\max_{\substack{a\md{m} \\\left(a,m\right)=1}} \left|\psi\left(x;m,a\right)-\frac{x}{\phi\left(m\right)}\right|.
\end{equation}
Now notice that by the definition of $\rho_n\left(m\right)$ one trivially has $\rho_n\left(m\right)\leq m$. Thus inequality \eqref{abcd} becomes
\begin{equation}
\left|\mathcal{E}\right|	\leq	\sum_{m \leq \sqrt{x}/\left(\log x\right)^{A} }
\max_{\substack{a\mdsub{m} \\ \left(a,m\right)=1}} 
\left|\psi\left(x;m,a\right)-\frac{x}{\phi\left(m\right)}\right|
\ll_A
\frac{x}{\left(\log x\right)^{A-5}},
\end{equation}
where a use of Theorem~\ref{Bombieri--Vinogradov} has been made. For the other term we get
\begin{equation}
\mathcal{M}	=	\sum_{m \leq \sqrt{x}/\left(\log x\right)^{A} }\frac{\mu\left(m\right)}{m}  \frac{\rho_n\left(m\right)}{\phi\left(m\right)}.
\end{equation}
We will show that this series converges. To that end, let us bound the tail of the series as follows. Using the inequality 
\begin{equation}
\phi\left(m\right)	\gg \frac{m}{\log m},
\end{equation}
valid for all $m\geq 2$, we deduce that
\begin{equation}
\left|	\sum_{m > \sqrt{x}/\left(\log x\right)^{A} }\frac{\mu\left(m\right)}{m}  \frac{\rho_n\left(m\right)}{\phi\left(m\right)}\right|	
\ll	
\sum_{m > \sqrt{x}/\left(\log x\right)^{A} } 
\mu^2(m)
\frac{\rho_n\left(m\right)}{m^2}\log m.
\end{equation}
By part~\ref{422} of Lemma~\ref{omeg}
\begin{equation}
\sum_{m > \sqrt{x}/\left(\log x\right)^{A} }   
\mu^2(m)
\frac{\rho_n\left(m\right)}{m^2}\log m
\ll_n
\frac
{\left(\log\left(
\frac{\sqrt{x}}{\left(\log x\right)^A}\right)
\right)^{\tau(n)+1}}
{\frac{\sqrt{x}}{\left(\log x\right)^A}}.
\end{equation}
This in turn is bounded by
\begin{equation}
\frac{\left(\log x\right)^{\tau(n)+A+1}}{\sqrt{x}},
\end{equation}
which tends to 0 as $x \to \infty$.
We may therefore write 
\begin{equation}
\mathcal{M}=
\sum_{m=1}^{\infty}\frac{\mu\left(m\right)}{m}
\frac{\rho_n\left(m\right)}{\phi\left(m\right)}
\ -
\sum_{m>\sqrt{x}/\left(\log x\right)^A}
\frac{\mu\left(m\right)}{m}
\frac{\rho_n\left(m\right)}{\phi\left(m\right)},
\end{equation}
which, by the preceding bound, equals
\begin{equation}
\sum_{m=1}^{\infty}\frac{\mu\left(m\right)}{m}
\frac{\rho_n\left(m\right)}{\phi\left(m\right)}+
O_n\left(\frac{\left(\log x\right)^{\tau(n)+A+1}}{\sqrt{x}}\right).
\end{equation}
Notice that the function $\frac{\mu\left(m\right)}{m}\frac{\rho_n\left(m\right)}{\phi\left(m\right)}$
is multiplicative, being the product of multiplicative functions.
We may thus use Euler products to deduce that 
\begin{equation}
\sum_{m=1}^{\infty}
\frac{\mu\left(m\right)}{m}
\frac{\rho_n\left(m\right)}{\phi\left(m\right)}=
\prod_{p}
\left(
\sum_{k=0}^{\infty}
\frac{\mu\left(p^k\right)}{p^k}
\frac{\rho_n\left(p^k\right)}{\phi\left(p^k\right)}
\right).
\end{equation} 
Recall the definition of $\p_n$ (equation~\eqref{product}), as well as the fact that
$\mu\left(p^k\right)=0$ for $k\geq 2$. 
Hence
\begin{align}
\sum_{m=1}^{\infty}
\frac{\mu\left(m\right)}{m}
\frac{\rho_n\left(m\right)}{\phi\left(m\right)}=&
\prod_{p}
\left(1+\frac{\mu\left(p\right)}{p}\frac{\rho_n\left(p\right)}{\phi\left(p\right)}
\right) \\
=&\prod_{p}\left(1-\frac{\gcd\left(p-1,n\right)}{p\left(p-1\right)}\right)
\\
=& n \p_n,
\end{align}
where the second equality follows from Lemma~\ref{rho}.
\end{proof}
We may now combine Lemma~\ref{psi}, Lemma~\ref{range}, and Lemma~\ref{asympt} to get that for any fixed $n \in \N$ and $A>0$ one has 
\begin{equation}\label{nine}
\Psi_n\left(x\right)=n \p_n x+O_{n,A}\left(\frac{x}{\left(\log x\right)^{A-5}}\right),
\end{equation}
for all $x \geq 2$.
\begin{proof}[Proof of Theorem~\ref{main}~~\ref{111}]

Using \eqref{nine} and the first part of Lemma~\ref{gyros} with 
\begin{equation}
b_k = \Lambda(k)\dfrac{\phi(k^n-1)}{k^n-1}
\end{equation}
we get
\begin{equation}\label{tena}
\sum_{p \leq x}\frac{\phi(p^n-1)}{p^n-1} \log p=n \p_n x+O_{n,A}\left(\frac{x}{(\log x)^{A-5}}\right).
\end{equation}
Inserting \eqref{tena} into the second part of Lemma~\ref{gyros} with $b_k=\frac{\phi(k^n-1)}{k^n-1}$ gives
\begin{equation}\label{ten}
\sum_{p \leq x}\frac{\phi(p^n-1)}{p^n-1} =n \p_n \li(x)+O_{n,A}\left(\frac{x}{(\log x)^{A-5}}\right).
\end{equation}
Using the first part of Lemma~\ref{gyros} with $b_k=n \p_n(k)$ gives
\begin{equation}\label{tent}
\sum_{q \leq x}\p_n(q)=\frac{1}{n}\sum_{p \leq x}\frac{\phi(p^n-1)}{p^n-1} +O_{n}\left(\frac{\sqrt{x}}{\log x}\right),
\end{equation}
which, when combined with \eqref{ten}, yields
\begin{equation}\label{ten22}
\sum_{q \leq x}\p_n(q)=\p_n\li(x)+O_{n,A}\left(\frac{x}{(\log x)^{A-5}}\right).
\end{equation}
Recall that $Q(x)$ denotes the number of prime powers that are at most $x$. To finish the proof it suffices to notice that by the Prime Number Theorem in the form 
\begin{equation}
\pi(x)=\li(x)+O_A\left(\frac{x}{(\log x)^{A}}\right)
\end{equation}
and the first part of Lemma~\ref{gyros} with $b_k=1$, we have
\begin{equation}\label{xm}
Q(x)		=	\pi(x)	+	O\left(\frac{\sqrt{x}}{\log x}\right),
\end{equation}
which implies that
\begin{equation}
Q(x)		=	\li(x)		+	O_A\left(\frac{x}{(\log x)^{A}}\right).
\end{equation}
Inserting this in \eqref{ten22} and using $\li(x)\gg \frac{x}{\log x}$, valid for all $x \geq 3$, yields 
\begin{equation}\label{last}
\frac{1}{Q(x)}	\sum_{q\leq x}	\p_n(q)	=	\p_n	+	O_{n,A}\left(	\frac{1}{(\log x)^{A-6}}		\right),
\end{equation}
for all $A>0$ and $x \geq 3$. 
Since once is allowed to use any positive value for $A$,
we can use $A+6$ instead. This may increase the dependence of the implied constant on $A$
but doesn't affect the validity of Theorem~\ref{main}.
We therefore conclude that for any positive $A$, the error term is
$O_{n,A}\left(	\frac{1}{(\log x)^{A}}		\right)$,
thus concluding the proof.
\end{proof}
The error term in \eqref{last} can be improved conditionally on the Generalised Riemann Hypothesis. Indeed, if one is to assume GRH for all $L$--functions of any modulus, then one can obtain
\begin{equation}
\psi(x;m,a)=\frac{x}{\phi(m)}+O\left(\sqrt{x}(\log x)^2\right),
\end{equation}
with an absolute implied constant,
for all $ m\geq 1$, $x\geq 3$, as shown in~\cite[Corollary 13.8]{mova}. Using this in place of Theorem~\ref{Bombieri--Vinogradov} one can follow the steps in the proof of Theorem~\ref{main} to prove
\begin{equation}
\frac{1}{Q(x)}	\sum_{q\leq x}	\p_n(q)	=	\p_n	+	O_n\left(	\frac{(\log x)^{\tau(n)+2}}{\sqrt{x}}		\right),
\end{equation}
for all $x\geq 3$.

The following lemma essentially contains the proof of both part~\ref{112}~and~\ref{113} of Theorem~\ref{main}.
It is proved via introducing multiplicative indices in the sum and then applying Lemma~\ref{nerdos}.
\begin{lemma}\label{phi}
For all naturals $r \geq 1$ and for all $x \in \R_{>1}$ one has 
\begin{equation}
\sum_{n \leq x}\frac{\phi(p^{rn}-1)}{p^{rn}-1}=\p(p,r)x+O\left(\log\left(x r \log p\right)\right).
\end{equation}
\end{lemma}
\begin{proof}
In view of~\eqref{penguin} we can write
\begin{equation}
\sum_{n \leq x}\frac{\phi(p^{rn}-1)}{p^{rn}-1} = \Osum_{m \leq p^{rx}}
\frac{\mu(m)}{m}
\sum_{\substack{n\leq x\\p^{rn}\equiv 1 \md{m}}}1.
\end{equation}
The condition 
$p^{rn}\equiv 1 \md{m}$
is equivalent to
\begin{equation}
\frac{\ell_p(m)}{\gcd\left(\ell_p(m),r\right)}\bigg|n.
\end{equation}
Grouping terms according to the value of the order, the double sum in the right-hand-side of the above equation is seen to equal
\begin{equation}
\sum_{k \leq r x}\left[\frac{x}{k}\gcd(k,r)\right]\left(\Osum_{\substack{m \in \N \\ \ell_p(m)=k}}\frac{\mu(m)}{m}\right).
\end{equation}
Using $[t]=t+O(1)$, valid for all $t \in \R_{\geq 0}$, we see that this is 
\begin{equation}
x\sum_{k \leq r x}\frac{\gcd(k,r)}{k}\left(\Osum_{\substack{m \in \N \\ \ell_p(m)=k}}\frac{\mu(m)}{m}\right)+O\left(\sum_{k \leq r x}\Osum_{\substack{m \in \N \\ \ell_p(m)=k}}\frac{1}{m}\right).
\end{equation}
Part~\ref{441} of Lemma~\ref{nerdos} implies that the error term is 
$\ll\log(rx\log p)$.
Recall the definition of $\p(p,r)$, stated in~\eqref{ff}. The
inequality $\gcd(k,r)\leq r$ implies
that the main term above equals
\begin{equation}
x \p(p,r)+O\left(xr\sum_{k>rx}\frac{1}{k}\Osum_{\substack{m \in \N \\ \ell_p(m)=k}}\frac{1}{m}\right).
\end{equation}
Using part~\ref{442} of Lemma~\ref{nerdos}
to handle the above error term
concludes our proof.
\end{proof}
\begin{proof}[Proof of Theorem~\ref{main}~~\ref{112}]
Recall the definition of $\p_n(p^r)$ in equation~\eqref{densityf}. Using Lemma~\ref{phi} yields
\begin{align}
\sum_{r \leq x}\p_n(p^r)		&=\frac{1}{n}\sum_{r \leq x}\frac{\phi(p^{rn}-1)}{p^{rn}-1}\\
						&=\frac{\p(p,n)}{n} x+O\left(\frac{\log\left(x n \log p \right)}{n}\right),
\end{align}
which proves the assertion of Theorem~\ref{main}~~\ref{112}.
\end{proof}
\begin{proof}[Proof of Theorem~\ref{main}~~\ref{113}]
Define for $x\geq 1$, $r \in \N$, and $p$ a prime
\begin{equation}
E(x,p,r)\coloneqq\sum_{n \leq x} \frac{\phi(p^{rn}-1)}{p^{rn}-1}-x\p(p,r),
\end{equation}
so that Lemma~\ref{phi} is equivalent to 
\begin{equation}\label{44}
E(x,p,r)\ll \log(xr\log p).
\end{equation}
Using partial summation one sees that for any $x \geq 1$,
\begin{equation}
\sum_{n \leq x}\p_n(p^r)=\frac{x \ \p(p,r)+E(x,p,r)}{x}+\int_1^x\frac{u \ \p(p,r)+E(u,p,r)}{u^2}\mathrm{d}u.
\end{equation}
Part~\ref{443} of Lemma~\ref{nerdos} combined with \eqref{44} shows that this equals
\begin{equation}
\p(p,r)\log x +O\left(\tau( r)\log \left(r \log p\right)\right).
\end{equation}
The proof is complete.
\end{proof}
\section{Proof of Theorem~\ref{main2}}\label{mainproof2}
We begin by stating two definitions that we will adhere to during the ensuing proofs.
\begin{definition}
A function $f:\N \to \R$ is called \textit{strongly--additive} if it satisfies
\begin{equation}
f\left(\prod_{p^{\lambda}\|n}p^\lambda\right)=\sum_{p^{\lambda}\|n}f\left(p\right)
\end{equation}
for all $n \in \N$.
\end{definition}
\begin{definition}
Let $b_n$ be a sequence of real numbers and $x,z \in \mathbb{R}$.
We define the frequencies
\[\nu_x\left(
p;b_p \leq z
\right):=\frac{|\{p\leq x: p \ \text{is prime}, b_p\leq z\}|}{\pi(x)},
\]
where $\pi(x)$ is the number of primes below $x$.
\end{definition}
Let $f$ be a strongly--additive function and let $g \in \Z[x]$. Define for each $m \in \N$
\begin{equation}
\rho_g(m)\coloneqq \left|\left\{a \in [0,m):g(a)\equiv 0 \md{m},\,\ \gcd(a,m)=1\right\} \right|
\end{equation}
and notice that this is a generalisation of the function $\rho_n$
defined at the beginning of section~\ref{lemmata}.
The next theorem can be found in \cite{tole}.
\begin{theorem}\label{toll}
Let $f$ and $g$ be as above.
Assume that $\rho_g( p) f( p)\to 0$ as $p \to \infty$ and that each of the following three series converges
\begin{equation}
\sum_{\left|f( p)\right|>1}\frac{\rho_g( p)}{p-1},
\sum_{\left|f( p)\right|\leq 1}\frac{\rho_g( p)f( p)}{p-1},
\sum_{\left|f( p)\right| \leq 1}\frac{\rho_g( p)f^2( p)}{p-1}.
\end{equation}
Then the frequencies $\nu_x \left(p;f \left( \left|g( p)\right|\right)\leq z\right)$ converge to a limiting distribution as $x \to \infty$. 
Furthermore, the said limiting distribution is continuous if and only if the series
\begin{equation}
\sum_{f( p)\neq 0}\frac{\rho_g( p)}{p-1}
\end{equation}
diverges.
\end{theorem}
\begin{proof}[Proof of Theorem~\ref{main2}~~\ref{main2i}]
We shall use Theorem~\ref{toll} to prove that for fixed $n \in \N$ the frequencies 
\begin{equation}
\nu_x\left(p;\frac{\phi(p^n-1)}{p^n-1}\leq z\right)
\end{equation}
converge to a limiting distribution as $x \to \infty$.
Define the strongly--additive function $f(k) \coloneqq \log \frac{\phi(k)}{k}$ and notice that the Taylor expansion of the logarithm 
implies that for any prime $p$, 
\begin{equation}
f( p)=\log \left(1-\frac{1}{p} \right) = - \frac{1}{p}+O\left(\frac{1}{p^2}\right).
\end{equation}
Define the polynomial $g_n(x) \coloneqq x^n-1$ and notice that by Lemma~\ref{rho}
\begin{equation}
\rho_{g_n}( p)=\gcd (p-1,n)\leq n
\end{equation}
for all primes $p$. Therefore
\begin{equation}
\rho_{g_n}( p) f( p) \ll \frac{n}{p} \to 0
\end{equation}
as $p \to \infty$.
We proceed to show that each of the three series in Theorem~\ref{toll} converge.
First, notice that $f( p) \in \left[-\log 2,0\right)$ for all primes $p$, hence the first series contains no terms. 
Regarding the second series one has
\begin{align}
\sum_{\left|f( p)\right|\leq 1}\left|\frac{\rho_{g_n}( p)}{p-1}f(p )\right|		&=	\sum_{p}\frac{\gcd(n,p-1)}{p-1}\left(\frac{1}{p}+O\left(\frac{1}{p^2}\right)\right)	\\ 
															&\ll	n \sum_{p}\frac{1}{p^2} 												\\
															&<	\infty,
\end{align}
thus the series is convergent.
Similarly
\begin{align}
\sum_{\left|f( p)\right|\leq 1}\left|\frac{\rho_{g_n}( p)}{p-1}\ f^2(p )\right|		&\ll n \sum_{p}\frac{1}{p^3} \\ 
															&< \infty,
\end{align}
and the third series converges as well. To conclude the proof of the theorem we observe that the limiting distribution is continuous due to 
\begin{align}
\sum_{\left|f( p)\right|\neq 0}\left|\frac{\rho_{g_n}( p)}{p-1}\right|			&=	\sum_{p}\frac{\gcd (n,p-1)}{p-1} 	\\
															&	\geq \sum_{p}\frac{1}{p-1} 		\\ 
															&= 	\infty.
\end{align}
Now notice that \eqref{xm}
implies that for each $x \geq 1$, $z \in \R$,
\begin{equation}
\nu_x\left(p;\frac{\phi(p^n-1)}{p^n-1}\leq z\right) = \nu_x\left(q;\frac{\phi(q^n-1)}{q^n-1}\leq z\right)+O\left(\frac{1}{\sqrt{x}}\right).
\end{equation}
Thus letting $x\to \infty$ shows that the limiting distribution of
\begin{equation}
\nu_x\left(p;\frac{\phi(p^n-1)}{p^n-1}\leq z\right)
\end{equation}
is equal to the limiting distribution of 
\begin{equation}
\nu_x\left(q;\frac{\phi(q^n-1)}{q^n-1}\leq z\right).
\end{equation}
The proof is now complete.
\end{proof}
The following theorem, stated in~\cite[Chapter III.2, Theorem 2]{tene}, provides a general criterion that ensures the existence of a limiting distribution. 
Before presenting it we need the following definition.
\begin{definition}
Let $A \subseteq \N$. The density of $A$ is defined as
\begin{equation}
\mathbf{d}(A)		\coloneqq		\lim_{x \to \infty}\frac{ \left| \left\{n \leq x: n \in A\right\} \right|}{x},
\end{equation}
provided that the limit exists, and the upper density of $A$ as
\begin{equation}
\mathbf{\overline{d}}(A)	\coloneqq		\limsup\limits_{x \to \infty}\frac{\left|\left\{n \leq x: n \in A\right\} \right|}{x}.
\end{equation}
\end{definition}
\begin{theorem}\label{tenn}
Let $f:\N\to \R$ be a function and suppose that for any $\epsilon>0$ there exists a function 
$\alpha_\epsilon(n):\N	\to \N$
having the following properties:
\begin{enumerate}
\item		\label{641}	$\lim\limits_{\epsilon \to 0} 	\limsup\limits_{T \to \infty}\mathbf{\overline{d}}\left\{n:\alpha_\epsilon(n)>T\right\}=0,$				\\
\item		\label{642}	$\lim\limits_{\epsilon \to 0}	\mathbf{\overline{d}}\left\{n:\left|f(n)-f\left(\alpha_\epsilon(n)\right)\right|>\epsilon\right\}=0$, \text{and}	\\	
\item		\label{643}	for each $\alpha \geq 1$ the density $\mathbf{d} \left\{n:\alpha_\epsilon(n)=\alpha\right\}$ exists.
\end{enumerate}
Then the frequencies $\nu_x\left(n;f(n)\leq z\right)$ converge to a limiting distribution as $x \to \infty$.
\end{theorem}
We shall use this theorem to prove the following lemma.
\begin{lemma}\label{tenenbaum}
Fix $c \in \N$ and a prime $\eta$. Then the frequencies 
\begin{equation}
\nu_x\left(k;\frac{\phi(\eta^{ck}-1)}{\eta^{ck}-1}\leq \omega\right)
\end{equation}
converge to a limiting distribution as $x \to \infty$.
\end{lemma}
\begin{proof}
Define for $k \in \N$
\begin{equation}
f(k) \coloneqq \log \frac{\phi(\eta^{ck}-1)}{\eta^{ck}-1} = \sum_{p \mid \eta^{ck}-1}\log \left(1-\frac{1}{p} \right)
\end{equation} 
and for each fixed $\epsilon>0$
\begin{equation}
\alpha_\epsilon(k)\coloneqq\prod_{\substack{p^\lambda \| k\\p\leq y}}p^\lambda
\end{equation}
where
\begin{equation}
y\coloneqq y(\epsilon) = \max\left\{\eta,\exp\left(\frac{c\log \eta}{\epsilon^2}\right)\right\}.
\end{equation}
It is not difficult to verify that with this choice of $y$ the validity of $y>\epsilon^{-2}$ is guaranteed, so that properties~\ref{641}~and~\ref{643} of Theorem~\ref{tenn} hold as 
in~\cite[Ex. 1, p. 295]{tene}. 
In order to verify the validity of property~\ref{642}, let us begin by noticing that since $\alpha_\epsilon(k)$ is a divisor of $k$, we get by the inequality 
$\log\left(1-\frac{1}{p}\right)\ll\frac{1}{p}$, valid for each prime~$p$, that
\begin{equation}\label{guiness}
\left|f(k)-f\left(\alpha_\epsilon(k)\right)\right|	\ll	\sum_{\substack{p \mid \eta^{ck}-1 \\ p>y}}\frac{1}{p}.
\end{equation}
Using \eqref{guiness} and the fact that 
$\ell_{\eta}( p) | ck$
implies
$\ell_{\eta}( p)/\gcd(\ell_{\eta}( p),c) | k$,
we deduce that
\begin{align}
\sum_{k \leq x}\left|f(k)-f\left(\alpha_\epsilon(k)\right)\right|		&\ll		\sum_{p \in \left(y,  \eta^{c x} \right)}	\frac{1}{p}	\left|\left\{k\leq x: \ell_{\eta}( p) \mid ck\right\}\right| \\ 
													&\leq 	x \sum_{p >y}\frac{1}{p}\frac{\gcd\left(\ell_{\eta}( p),c\right)}{\ell_{\eta}( p)}.
\end{align}
In light of the inequalities
$\gcd\left(\ell_{\eta}( p),c\right)\leq c$
and $\log p< \ell_\eta( p ) \log \eta$, the last expression is seen to be at most
\begin{equation}
c x \log \eta
\sum_{p>y}\frac{1}{p \log p}.
\end{equation}
Now Lemma~\ref{pt} for $a=m=1$ combined with partial summation implies that
\begin{align}
\sum_{p>y}\frac{1}{p \log p}	&=-\frac{\log \log y + O\left( 1\right) }{\log y}	+ \int_{y}^{\infty}\frac{\log \log t + O\left( 1\right) }{t \log^2 t} \mathrm{d}t \\
						&\ll	\frac{1}{\log y}, 
\end{align}
which shows that
\begin{equation}
\sum_{k \leq x}		\left|f(k)-f\left(\alpha_\epsilon(k)\right)\right|	\ll	\frac{cx\log \eta}{\log y}.
\end{equation}
We may now use this inequality to deduce that
\begin{align}
\frac{1}{x}\left|\left\{k\leq x: \left|f(k)-f\left(\alpha_\epsilon(k)\right)\right| > \epsilon\right\}\right|	&\leq	\sum_{k \leq x}\frac{\left|f(k)-f\left(\alpha_\epsilon(k)\right)\right|}{\epsilon x}		\\
																	&\ll		\frac{c\log \eta}{\epsilon \log y}	\leq		\epsilon , 																						
\end{align}
by the definition of $y$. This establishes the validity of property~\ref{642} and therefore Theorem~\ref{tenn} applies and shows that the frequencies 
\begin{equation}
\nu_x\left(k;\frac{\phi(\eta^{ck}-1)}{\eta^{ck}-1}\leq \omega\right)
\end{equation}
converge to a limiting distribution as $x \to \infty$.
\end{proof}
Now the proof of part~\ref{main2ii} of Theorem~\ref{main2} follows
by setting
$\eta=p, c=n, k=r$ and $\omega=zn$ in Lemma~\ref{tenenbaum}.
The proof of part~\ref{main2iii} of Theorem~\ref{main2} follows
by setting
$\eta=p, c=r, k=n$ and $\omega=z$
in the same Lemma.
\begin{ack}
We thank T. Browning and R. A. Wilson for helpful comments and I. E. Shparlinski for bringing his paper~\cite{shpar} to our attention.
Both the anonymous referees have contributed substantially to improving an earlier version of this manuscript by providing us with 
numerous comments and suggestions. For this we duly thank them.
\end{ack}
\bibliographystyle{amsalpha}

\end{document}